\documentclass[a4,12pt]{amsart}
\usepackage{amssymb,amstext,amsmath,amscd,amsthm,amsfonts,enumerate,latexsym}
\usepackage{color}
\usepackage[dvipdfmx]{graphicx}
\usepackage[all]{xy}
\usepackage{extarrows}
\theoremstyle{plain}
\newtheorem{thm}{Theorem}[section]
\newtheorem*{thm*}{Theorem}
\newtheorem*{cor*}{Corollary}

\newtheorem{prop}[thm]{Proposition}
\newtheorem{lem}[thm]{Lemma}

\newtheorem*{claim*}{Claim}

\theoremstyle{definition}
\newtheorem{defn}[thm]{Definition}

\newtheorem{ques}[thm]{Question}

\newtheorem{remark}[thm]{Remark}

\theoremstyle{remark}

\numberwithin{equation}{thm}

\def\Ext{\operatorname{Ext}}

\def\Ker{\operatorname{Ker}}

\def\mod{\mathrm{mod}}

\def\a{\mathrm a}

\def\e{\mathrm{e}}
\def\m{\mathfrak m}
\def\n{\mathfrak n}

\def\q{\mathfrak q}

\newcommand{\rma}{\mathrm{a}}

\newcommand{\rmo}{\mathrm{o}}

\newcommand{\rmr}{\mathrm{r}}

\newcommand{\rmK}{\mathrm{K}}

\newcommand{\calR}{\mathcal{R}}

\newcommand{\fka}{\mathfrak{a}}

\newcommand{\fkm}{\mathfrak{m}}
\newcommand{\fkn}{\mathfrak{n}}

\newcommand{\fkM}{\mathfrak{M}}

\newcommand{\mapright}[1]{%
\smash{\mathop{%
\hbox to 1cm{\rightarrowfill}}\limits^{#1}}}

\newcommand{\mapleft}[1]{%
\smash{\mathop{%
\hbox to 1cm{\leftarrowfill}}\limits_{#1}}}

\begin{document}

\setlength{\baselineskip}{17pt}
\title[Contracted ideals]{On the almost Gorenstein property in the Rees algebras of contracted ideals}
\author{Shiro Goto}
\address{Department of Mathematics, School of Science and Technology, Meiji University, 1-1-1 Higashi-mita, Tama-ku, Kawasaki 214-8571, Japan}
\email{shirogoto@gmail.com}
\author{Naoyuki Matsuoka}
\address{Department of Mathematics, School of Science and Technology, Meiji University, 1-1-1 Higashi-mita, Tama-ku, Kawasaki 214-8571, Japan}
\email{naomatsu@meiji.ac.jp}
\author{Naoki Taniguchi}
\address{Department of Mathematics, School of Science and Technology, Meiji University, 1-1-1 Higashi-mita, Tama-ku, Kawasaki 214-8571, Japan}
\email{taniguchi@meiji.ac.jp}
\author{Ken-ichi Yoshida}
\address{Department of Mathematics, College of Humanities and Sciences, Nihon University, 3-25-40 Sakurajosui, Setagaya-Ku, Tokyo 156-8550, Japan}
\email{yoshida@math.chs.nihon-u.ac.jp}

\thanks{2010 {\em Mathematics Subject Classification.} 13H10, 13H15, 13A30}
\thanks{{\em Key words and phrases.} almost Gorenstein local ring, almost Gorenstein graded ring, Rees algebra, contracted ideal}
\thanks{The first author was partially supported by JSPS Grant-in-Aid for Scientific Research 25400051. The second author was partially supported by JSPS Grant-in-Aid for Scientific Research 26400054. The fourth author was partially supported by JSPS Grant-in-Aid for Scientific Research 25400050.}

\maketitle

\begin{abstract} 
The question of when the Rees algebra $\calR (I)= \bigoplus_{n \ge 0}I^n$ of $I$ is an almost Gorenstein graded ring is explored, where $R$ is a two-dimensional regular local ring and $I$ a contracted ideal of $R$.  By \cite{GMTY1} it is known that $\calR(I)$ is an almost Gorenstein graded ring for every integrally closed ideal $I$ of $R$. The main results of the present paper show that if $I$ is a contracted ideal with $\mathrm{o}(I) \le 2$, then $\calR(I)$ is an almost Gorenstein graded ring, while if $\mathrm{o}(I) \ge 3$, then $\calR(I)$ is not necessarily an almost Gorenstein graded ring, even though $I$ is a contracted stable  ideal. Thus both affirmative answers and negative answers are given.
\end{abstract}

\section{Introduction}\label{intro}
Let $(R,\m)$ be a two-dimensional regular local ring. The purpose of this paper is to study the problem of when the Rees algebras of contracted ideals in $R$ are almost Gorenstein graded rings. In \cite{GMTY1} the authors showed that for every integrally closed ideal $I$ of $R$ the Rees algebra $\calR(I)=\bigoplus_{n \ge 0}I^n$ of $I$ is an almost Gorenstein graded ring. Since every integrally closed $\m$-primary ideal is contracted and stable (\cite{Z}), it seems quite natural to expect a similar affirmative answer for contracted stable ideals also.  Curiously, this is not the case and the answer depends on the order $\mathrm{o}(I)$ of the contracted ideals $I$. If $\mathrm{o}(I) \le 2$, then $\calR(I)$ is an almost Gorenstein graded ring, while if $\mathrm{o}(I) \ge 3$, then $\calR(I)$ is not necessarily an almost Gorenstein graded ring, which we shall show in this paper. But before entering details, let us  recall the definitions of almost Gorenstein local/graded rings and some historical notes as well.

For the last sixty years commutative algebra has been concentrated mostly in the study of Cohen-Macaulay rings/modules. Several experiences in our researches give the impression that Gorenstein rings are rather isolated in the class of Cohen-Macaulay rings. Gorenstein local rings are defined by the finiteness of self-injective dimension. There is, however, a substantial gap between the conditions of the finiteness of self-injective dimension and the infiniteness of it. The notion of almost Gorenstein ring is an attempt to go beyond this gap or a desire to find a new class of Cohen-Macaulay rings which might be non-Gorenstein but still good, say the next to Gorenstein rings.

The notion of almost Gorenstein local ring in our sense dates back to the paper \cite{BF} of V. Barucci and R. Fr\"oberg in 1997, where they introduced the notion to one-dimensional analytically unramified local rings and developed a very beautiful theory of almost symmetric numerical semigroups. As their definition is not flexible enough for the analysis of analytically ramified local rings, in 2013  S. Goto, N. Matsuoka, and T. T. Phuong \cite{GMP} relaxed the restriction and gave the definition of almost Gorenstein local rings for arbitrary but still one-dimensional Cohen-Macaulay local rings, using the first Hilbert coefficients of canonical ideals. They constructed in \cite{GMP} numerous examples of almost Gorenstein local rings which are analytically ramified, extending several important results of \cite{BF}. However it might be the biggest achievement of \cite{GMP} that the paper prepared for higher dimensional definitions and opened the door led to a frontier. In fact in 2015 S. Goto, R. Takahashi, and N. Taniguchi \cite{GTT}  gave the following definitions of almost Gorenstein local/graded rings of higher dimension and started the theory.

\begin{defn}[The local case]\label{1.1}
Let $(A,\fkm)$ be a Cohen-Macaulay local ring of dimension $d$, possessing the canonical module $\rmK_A$. Then we say that $A$ is an almost Gorenstein local ring, if there exists an exact sequence
$$
0 \to A \to \rmK_A \to C \to 0
$$
of $A$-modules such that either $C = (0)$ or $C \ne (0)$ and $\mu_A(C) = \e^0_\fkm(C)$, where  $\mu_A(C)$ denotes the number of elements in a minimal system of generators of $C$ and $$\e^0_\fkm(C) = \lim_{n\to \infty}(d-1)!{\cdot}\frac{\ell_A(C/\fkm^{n+1}C)}{n^{d-1}}$$ denotes the multiplicity of $C$ with respect to the maximal ideal $\fkm$ (here $\ell_A(*)$ stands for the length).
\end{defn}

Let us explain a little more about Definition \ref{1.1}. Let $(A,\fkm)$ be a Cohen-Macaulay local ring of dimension $d$ and assume that $A$ possesses the canonical module $\rmK_A$. The condition of Definition \ref{1.1} requires that  $A$ is embedded into   $\rmK_A$ and even though $A \ne \rmK_A$, the difference $C = \rmK_A/A$ between $\rmK_A$ and $A$ is an Ulrich $A$-module (cf. \cite{BHU}) and behaves well. Here we notice that for every exact sequence$$
0 \to A \to \rmK_A \to C \to 0
$$
of $A$-modules, $C$ is a Cohen-Macaulay $A$-module of dimension $d-1$, provided $C \ne (0)$ (\cite[Lemma 3.1 (2)]{GTT}).

\begin{defn}[The graded case]\label{1.2}
Let $R = \sum_{n \ge 0}R_n$ be a Cohen-Macaulay graded ring such that $A = R_0$ is a local ring. Assume that $A$ is a homomorphic image of a Gorenstein local ring and let $\rmK_R$ denote the graded canonical module of $R$. We set $d = \dim R$ and $a = \a(R)$ the $a$-invariant of $R$. Then we say that $R$ is an almost Gorenstein graded ring, if there exists an exact sequence 
$$
0 \to R \to \rmK_R(-a) \to C \to 0
$$
of graded $R$-modules such that either $C = (0)$ or $C \ne (0)$ and $\mu_R(C) = \e^0_M(C)$, where $M$ denotes the graded maximal ideal of $R$. 
\end{defn}

In Definition \ref{1.2} suppose $C \ne (0)$. Then $C$ is a Cohen-Macaulay graded $R$-module and $\dim_RC  = d -1$.  As 
$
\e_M^0(C) = \lim_{n\to\infty}(d-1)!{\cdot}\frac{\ell_R(C/M^{n+1}C)}{n^{d-1}},
$
we get $\e_{MR_M}^0(C_M)= \e^0_M(C)$, so that $C_M$ is an Ulrich $R_M$-module. Therefore since $\rmK_{R_M}= \left[\rmK_R\right]_M$, $R_M$ is an almost Gorenstein local ring if $R$ is an almost Gorenstein graded ring. The converse is not true in general (\cite[Theorems 2.7, 2.8]{GMTY2}, \cite[Example 8.8]{GTT}).

Let us  now consider a more specific setting where $(R,\m)$ is a two-dimensional regular local ring and $I$ is an $\fkm$-primary ideal of $R$. For simplicity we always assume that the field $R/\fkm$ is infinite. Then we say that $I$ is a contracted ideal of $R$ if 
$$
I {\cdot}R \left[\frac{\fkm}{x}\right] \cap R = I
$$
for some (actually, for every general) element $x \in \fkm \setminus \fkm^2$, where $\frac{\fkm}{x} = \left\{\frac{a}{x} \mid a \in \fkm\right\}$ which is considered inside of the quotient field $\mathrm{Q}(R)$ of $R$. Let $$\calR(I) = \bigoplus_{n \ge 0}I^n$$ be the Rees algebra of $I$ and set $\fkM = \m {\cdot} \calR (I) +[\calR (I)]_+$, the graded maximal ideal of $\calR (I)$. Let us fix  a minimal reduction $Q = (a,b)$ of $I$. Then $I$ is said to be stable, if $I^2 = QI$. This condition is equivalent to saying that $\calR(I)$ is a Cohen-Macaulay ring (\cite{GS}).

With this notation we have the following striking characterization of contracted stable ideals.

\begin{thm}[\cite{V2}]\label{V2}
The following conditions are equivalent.
\begin{enumerate}[$(1)$]
\item $\calR (I)$ is a Cohen-Macaulay ring and $\calR (I)_\fkM$ possesses maximal embedding dimension in the sense of \cite{S}.
\item $I$ is a contracted stable ideal of $R$.
\end{enumerate}
\end{thm}

The motivation of the present research has come from the question of when $\calR(I) = \bigoplus_{n \ge 0}I^n$ is an almost Gorenstein graded ring. In \cite{GMTY1} the authors showed that if $I$ is an integrally closed ideal, then  $\calR(I)$ is an almost Gorenstein graded ring. Since every integrally closed $\fkm$-primary ideal $I$ is contracted and stable (\cite{Z}; see also \cite{G, Hu}) and since $\calR (I)$ is a good Cohen-Macaulay ring as Theorem \ref{V2} guarantees, it seems quite natural to expect that $\calR(I)$ is an almost Gorenstein graded ring also for every contracted stable ideal $I$. This is, however, not true in general, as we shall report in Theorem \ref{4.2} of the present paper and the following two theorems \ref{1.3} and \ref{1.4}.  For each ideal $\fka$ of $R$ let $$\mathrm{o}(\fka) = \max\{n \in \Bbb Z \mid \fka \subseteq \fkm^n \}$$ and call it the order of $\fka$.

\begin{thm}\label{1.3}
Let $(R,\fkm)$ be a two-dimensional regular local ring with infinite residue class field. Let $I$ be a contracted $\fkm$-primary ideal of $R$ with $\mathrm{o}(I) \le 2$. Then $I$ is a stable ideal of $R$, whose Rees algebra $\calR(I)=\bigoplus_{n \ge 0}I^n$ is an almost Gorenstein graded ring. 
\end{thm}

\noindent
If $\mathrm{o}(I) = 1$, then $I = Q~(=(a,b))$ is a parameter ideal of $R$, so that $\calR(I) = R[X,Y]/(aY-bX)$ is a Gorenstein ring, where $R[X,Y]$ denotes the polynomial ring. Therefore our interest in Theorem \ref{1.3} is the case $\mathrm{o}(I) = 2$. The simplest example of contracted ideals of order $2$ which are not integrally closed is the ideal $I = (x^2, xy^4, y^5)$, where $x,y$ is a regular system of parameters of $R$.

\begin{thm}\label{1.4} 
Let $(R,\fkm)$ be a two-dimensional regular local ring with infinite residue class field and write $\fkm = (x,y)$ with $x, y \in \fkm$. Choose integers $n$, $\alpha$, and $\beta$ so that  
\begin{center}
$0 < \alpha < \beta < n$, \ $n < \alpha + \beta$, \ $n + \alpha < 2\beta$, \ and \ $\beta < 2 \alpha$.
\end{center} Set $I=(x^3, x^2y^{\alpha}, xy^\beta, y^n)$ and $Q=(x^3, y^n)$. Then  $I$ is a contracted ideal of $R$ with $I^2 = QI$ and $\mathrm{o}(I) = 3$,  but $\calR(I)$ is not an almost Gorenstein graded ring.
\end{thm}

\noindent
The simplest example is the ideal $I=(x^3,x^2y^3,xy^5,y^6)$.

We shall prove Theorem \ref{1.3} (resp. Theorem \ref{1.4}) in Section 3 (resp. Section 4). For the proofs we need some preliminaries on contracted ideals and Rees algebras as well, which we will summarize in Section 2. The proof given in \cite{GMTY1} for the case where the ideals $I$ are integrally closed heavily depends on the existence of joint reductions with joint reduction number $0$ (\cite{V1}). To prove Theorem \ref{1.3} one cannot use this result, since it is no longer true for contacted ideals. Instead, we have to evolve rather fine arguments in Section 3, applying the affirmative result for integrally closed ideals also. Section 5 is devoted to the analysis of non-contracted ideals of order $3$.

In what follows, unless otherwise specified, let $(R,\m)$ be a two-dimensional regular local ring with infinite residue class field. For each finitely generated $R$-module $M$ let $\mu_R(M)$ (resp. $\ell_R(M))$ denote the number of elements  in a minimal system of generators (resp. the length) of $M$.




\section{Preliminaries}
Let $I$ be an $\m$-primary ideal of $R$ and choose a parameter ideal $Q$ of $R$ so that $Q$ is a reduction of $I$. Hence $Q \subseteq I$ and $I^{r+1} = QI^r$ for some $r \ge 0$.

To begin with, let us recall the following characterization of contracted ideals. This follows from the results of \cite[Appendix 5]{ZS} and \cite{HSa, L, R}. Consult \cite{Hu}, \cite{HSa}, and \cite[Section 14]{SH} for detailed proofs. Here we notice that $\fkm I : x = I$ for every general element $x \in \fkm \setminus \m^2$, once $I$ is a contracted ideal of $R$.

\begin{prop}\label{2.2}
The following conditions are equivalent.
\begin{enumerate}[$(1)$]
\item $I$ is a contracted ideal of $R$.
\item $I$ is $\fkm$-full, that is $\m I : x =I$ for some $x \in \fkm$. 
\item $I: x = I: \m$ for some $x \in \fkm \setminus \fkm^2$.
\item $\mu_R(I) = \mathrm{o}(I) + 1$.
\end{enumerate}
When this is the case, one has $x \not\in \m^2$ for the element $x$ in condition $(2)$.
\end{prop}

The following result might be known. We include a brief proof for the sake of completeness.

\begin{lem}\label{2.3} The following conditions are equivalent.
\begin{enumerate}[$(1)$]
\item $I$ is a contracted ideal of $R$.
\item $\m I = \m a + x I$  for some $x \in \m$ and $a \in I$.
\end{enumerate}
When this is the case, one can choose the element $a \in I$ so that $a \in Q \setminus \m Q$.
\end{lem}

\begin{proof}
(1) $\Rightarrow$ (2) 
By Proposition \ref{2.2} $\m I : x = I$ for some $x \in \m \setminus \m^2$. We set $D=R/(x)$. Then since $D$ is a discrete valuation ring and $Q$ is a reduction of $I$, $ID = QD$. Let us write $QD = aD$ with $a \in Q \setminus \m Q$. Then  $I \subseteq (a, x)$, so that 
$$
\m I \subseteq [\m a + (x)] \cap \m I = \m a + x (\m I : x) = \m a + x I.
$$
Hence $\m I = \m a + x I$.

(2) $\Rightarrow$ (1)
Notice that $a,x$ is a system of parameters of $R$, since $\fkm I \subseteq (a,x)$. Let $h \in \m I : x$ and write $x h = a f + x g$ with $f \in \m$ and $g \in I$. Then since the sequence $a, x$ is  $R$-regular, we get $h -g \in (a) \subseteq I$, whence $\m I : x \subseteq I$. Thus  $I$ is a contracted ideal by Proposition \ref{2.2}.
\end{proof}

We apply Lemma \ref{2.3} to get the following, where assertion $(1)$ is known by \cite{Z}. Let us include a brief proof in our context.

\begin{prop}\label{2.4} Suppose that $I$ is a contracted ideal of $R$. 
\begin{enumerate}[$(1)$]
\item Let $J$ be an $\fkm$-primary ideal of $R$ and suppose that $J$ is contracted. Then $IJ$ is also  contracted.
\item  Suppose that $Q \ne I$. Then $Q:I$ is a contracted ideal of $R$.
\end{enumerate}
\end{prop}

\begin{proof}
(1) We write $\fkm I = \fkm a + x I$ and $\fkm J = \fkm b + x J$ with $x \in \fkm$ and $a \in I$ and $b \in J$. Then
\begin{center}
$\m (IJ) = (\m a+xI)J = \m (aJ) + x(IJ) =a (\m b + xJ) + x(IJ) = \m(ab) + x(IJ)$.
\end{center}
Hence $IJ$ is a contracted ideal of $R$ by Lemma \ref{2.3}.

(2) Choose $x \in \m$ and $a \in Q \setminus \m Q$ so that $\m I = \m a + x I$.  Then $x \not\in \fkm^2$; otherwise $\fkm I = \fkm a \subseteq (a)$ by Nakayama's lemma, which is impossible because $\fkm I$ is an $\fkm$-primary ideal. We set $J = Q:I$. Then
$$
J: \m = (Q: I) : \m = Q: \m I = Q: xI = (Q: I) : x =J: x
$$
whence $J$ is a contracted ideal by Proposition \ref{2.2}.
\end{proof}

Let us consider the Rees algebra of $I$. We set 
$$
\calR = \calR(I) = R[It] \subseteq R[t]
$$
where $t$ denotes an indeterminate over $R$. Let $\fkM = \m \calR + \calR_+$ be  the unique graded maximal ideal of $\calR$.  We assume that $I^2 = QI$ and set $J=Q: I$.  Hence $\calR$ is a Cohen-Macaulay ring (\cite{GS}) with  $\dim \calR = 3$ and $\rma(\calR) = -1$, whose graded canonical module $\rmK_\calR$ is given by the following.

\begin{prop}[{\cite{GMTY1, U}}]\label{2.4a}
$
\rmK_{\calR}(1) = J \calR.
$
\end{prop}

One of the methods to see  whether $\calR$ is an almost Gorenstein graded ring is the following.

\begin{thm}$($\cite[Theorem 2.3]{GMTY1}$)$\label{2.5}
 The following conditions are equivalent.
\begin{enumerate}[$(1)$]
\item There exists an exact sequence
$$
0 \to \calR \to \rmK_{\calR}(1) \to C \to 0
$$
of graded $\calR$-modules such that $$\fkM C = (\xi, \eta)C$$ for some homogeneous elements $\xi, \eta$ of $\fkM$.
\item There exist elements $f\in \m$, $g \in I$, and $h \in J$ such that
$$
IJ = gJ + I h \ \ \ \text{and} \ \ \ \m J = f J + \m h.
$$
\end{enumerate}
When this is the case, $\calR$ is an almost Gorenstein graded ring. 
\end{thm}

For each Cohen-Macaulay local ring $(A,\n)$ of dimension $d$ we denote by $$\rmr(A) = \ell_A(\Ext_A^d(A/\fkn,A))$$ the Cohen-Macaulay type of $A$ (\cite[Definition 1.20]{HK}).

\begin{prop}\label{2.8}
Suppose that $I$ is a contracted ideal of $R$. Then $\rmo(I) = \rmo(J) + 1$.
\end{prop}

\begin{proof}
We may assume that $\rmo(I) >1$. We set $A =\calR_\fkM$ and $\n = \fkM\calR_\fkM$.  Let $n=\rmo(I)$. Then since $I$ is a contracted ideal and $I^2 = QI$, by Theorem \ref{V2} the Cohen-Macaulay local ring $A$ has maximal embedding dimesnion, that is 
$
\n^2= \q \n
$
for some parameter ideal $\q$ of $A$. Hence $\mu_A(\n) - 3 = \rmr (A)$ as $\dim A = 3$, while $\mu_A(\n) = \mu_R(I) + 2$ and $\rmr(A) = \mu_R(J)$ by Proposition \ref{2.4a} (see \cite[Korollar 6.11]{HK}). Thus  $\mathrm{o}(I) = \mathrm{o}(J) + 1$, because $\mu_R(I) = \mathrm{o}(I) + 1$ and $\mu_R(J) = \mathrm{o}(J) + 1$ by Propositions \ref{2.2} and \ref{2.4} (2).
\end{proof}

Closing this section, we note the following result, which shows the Rees algebras of certain contracted ideals of high order are always almost Gorenstein graded rings.

\begin{thm}\label{2.7}
Let $(R,\m)$ be a two-dimensional regular local ring with infinite residue class field and let $x,y$ be a regular system of parameters of $R$. Let $2 \le m \le n$ be integers and set $I = (x^m) + y^{n-m+1}\m^{m-1}$. Then $I$ is a stable contracted ideal with $\mathrm{o}(I) = m$ and $\calR(I)$ is an almost Gorenstein graded ring. If $n \ge 2m$, then $xy^{n-2}$ is integral over $I$ but $xy^{n-2} \not\in I$, so that $I$ is not integrally closed.
\end{thm}

\begin{proof} Since $\mu_R(I) = m+1$ and $\mathrm{o}(I) = m$, $I$ is a contracted ideal. We set $Q=(x^m, y^n)$ and $J = Q:I$. Then $Q \subseteq I$ and $I^2 = QI$. We have
$$
J = \bigcap_{i=1}^{m-1}\left[Q:x^i y^{n-i}\right] = \bigcap_{i=1}^{m-1} (x^{m-i}, y^i) = \m^{m-1}
$$
whence $\m J = \m y^{m-1} + x J$. Notice that
$$
IJ = y^n \m^{m-1} + x^m J \subseteq [(x^m) + y^{n-m+1}\m^{m-1}]y^{m-1} + x^m J = I y^{m-1} + x^m J
$$
and we get $IJ = I y^{m-1} + x^m J$. Thus  by Theorem \ref{2.5} $\calR(I)$ is an almost Gorenstein graded ring.
\end{proof}



\section{Proof of Theorem \ref{1.3}}

Let $(R, \m)$ be a regular local ring with $\dim R = 2$ and infinite residue class field. Let $I \subsetneq R$ be an $\m$-primary ideal of $R$. We choose a parameter ideal $Q=(a,b)$ of $R$ so that $Q \subseteq I$ and $I^2 = Q I$ and set $J=Q:I$. Suppose that $I$ is a contracted ideal with $\rmo(I)=2$. Our purpose is to prove Theorem \ref{1.3}.

By Propositions \ref{2.4} (2) and \ref{2.8}, $J$ is a contracted ideal with $\rmo(J) = 1$. Let us write 
$$
J=(f, g)
$$ 
with $f, g \in R$ and $f \not\in \m^2$. Since $Q=(a, b) \subseteq J=(f, g)$, we have the presentation
$$
\begin{pmatrix}
a \\
b
\end{pmatrix}
=
\begin{pmatrix}
\alpha & \beta \\
\gamma & \delta
\end{pmatrix}
\begin{pmatrix}
f \\
g
\end{pmatrix} \quad  (*)
$$
with $\alpha, \beta, \gamma, \delta \in R$. 
 Let  $$c = \det 
\begin{pmatrix}
\alpha & \beta \\
\gamma & \delta
\end{pmatrix} \quad  (**)
$$ and  we get $Q: c = J$ by \cite[Theorem 3.1]{H}, whence $I=Q+(c)=(a,b,c)$ as $I = Q:J$. Let
$$
\Bbb A={}^t
\begin{pmatrix}
g & \alpha & \gamma \\
-f & \beta & \delta
\end{pmatrix}.
$$
Then since $I$ is generated by the $2\times 2$ minors of the matrix $\Bbb A$, we get a minimal free resolution of $R/I$ of the form (\cite[Theorem 1.4.16]{BH})
$$
0 \to R^2 \overset{\Bbb A}{\longrightarrow} R^3 \overset{[\begin{smallmatrix}
c & -b &a\\ \end{smallmatrix}]}{\longrightarrow} R \to R/I \to 0.
$$

Let $S = R[T_1, T_2, T_3]$ be the polynomial ring and let 
$$
\varphi : S \to \calR = R[It]
$$
be the homomorphism of $R$-algebras defined by $\varphi(T_1) = ct$, $\varphi(T_2) = -bt$, and $\varphi(T_3) = at$, where $\calR = \calR(I)$  denotes the Rees algebra of $I$. We set 
$$
F_1 = g T_1 + \alpha T_2 + \gamma T_3 \ \ \text{and}\ \ F_2 = -f T_1 + \beta T_2 + \delta T_3.
$$ 
Then since $I^2 = QI$, by \cite[Theorem 4.1]{MP} we get
$$
\Ker \varphi =(F_1, F_2, F)
$$ 
for some $
F \in S_2
$, accounting the fact $c^2 \in QI$. On the other hand, since $\Ker \varphi$ is a perfect ideal of $S$ with grade $2$ and $\rmr(\calR) = 2$ by Proposition \ref{2.4a},  $\calR$ has a graded minimal $S$-free resolution of the following form 
$$
0 \to S(-2)\oplus S(-2) \overset{\Bbb{M}}{\longrightarrow} S(-2)\oplus S(-1) \oplus S(-1) \overset{[\begin{smallmatrix}
F&F_1&F_2\\
\end{smallmatrix}]}{\longrightarrow}  S \to \calR \to 0.
$$
Since $\rmK_S= S(-3)$, taking the $\rmK_S$-dual we get the following presentation of the graded canonical module $\rmK_{\calR}=[J\calR](-1)$ of $\calR$ (see Proposition \ref{2.4a}).

\begin{prop}\label{res}
$$
0 \to S(-3) \overset{{}^t[\begin{smallmatrix}
F&F_1&F_2\\
\end{smallmatrix}]}{\longrightarrow}  S(-1)\oplus S(-2) \oplus S(-2)  \overset{{}^t\Bbb{M}} {\longrightarrow} S(-1)\oplus S(-1) \overset{\sigma}{\longrightarrow} [J\calR](-1) \to 0.
$$
\end{prop}

Let us write
$$
{}^t \Bbb{M}
=
\begin{pmatrix}
\xi & \alpha_1 T_1 + \alpha_2 T_2 + \alpha_3 T_3  & \gamma_1 T_1 + \gamma_2 T_2 + \gamma_3 T_3 \\
\eta & \beta_1 T_1 + \beta_2 T_2 + \beta_3 T_3  & \delta_1 T_1 + \delta_2 T_2 + \delta_3 T_3 
\end{pmatrix}
$$
with $\xi, \eta, \alpha_i, \beta_i, \gamma_i, \delta_i \in R$.  We maintain this notation throughout this section.

\begin{lem}\label{3.1}
$J=(\xi, \eta)$.
\end{lem}

\begin{proof}
The theorem of Hilbert-Burch (\cite[Theorem 1.4.16]{BH}) applied to the ideal $\Ker \varphi =(F_1, F_2, F)$ of $S$ shows 
\begin{eqnarray*}
F_1 &=& (-u) [\xi (\delta_1 T_1 + \delta_2 T_2 + \delta_3 T_3) - \eta (\gamma_1 T_1 + \gamma_2 T_2 + \gamma_3 T_3 )] \quad \text{and}\\
F_2 &=& u [\xi (\beta_1 T_1 + \beta_2 T_2 + \beta_3 T_3) - \eta (\alpha_1 T_1 + \alpha_2 T_2 + \alpha_3 T_3 )]
\end{eqnarray*}
for some unit $u$ of $R$. Therefore
$$
\begin{pmatrix}
g \\
\alpha \\
\gamma
\end{pmatrix}
=(-u)\left[\xi
\begin{pmatrix}
\delta_1 \\
\delta_2 \\
\delta_3
\end{pmatrix}
- \eta
\begin{pmatrix}
\gamma_1 \\
\gamma_2 \\
\gamma_3
\end{pmatrix}
\right]
, ~~
\begin{pmatrix}
-f \\
\beta \\
\delta
\end{pmatrix}
=u \left[\xi
\begin{pmatrix}
\beta_1 \\
\beta_2 \\
\beta_3
\end{pmatrix}
- \eta
\begin{pmatrix}
\alpha_1 \\
\alpha_2 \\
\alpha_3
\end{pmatrix}
\right]
 \quad (***)
$$
whence $J =(f,g) \subseteq (\xi, \eta)$. Let $\{{\bf e}_i\}_{i = 1,2}$ denote the standard basis of $S \oplus S$ and set $f' = \sigma({\bf e}_1)$ and $g' = \sigma({\bf e}_2)$, where $$S(-1)\oplus S(-1) \overset{\sigma}{\longrightarrow} [J\calR](-1)$$ denotes the homomorphism given in Proposition \ref{res}. Then we get $J=(f', g')$ and hence $f',g'$ forms an $R$-regular sequence.  Because $\sigma{\cdot} {}^t\Bbb M = 0$ in Proposition \ref{res}, we have
$$
\xi f' + \eta g' = 0
$$ 
whence
$$
\xi = g' z, ~~\eta = -f' z
$$
for some $z \in R$. Thus $(\xi, \eta) \subseteq (f', g') = J$, so that $J = (\xi,\eta)$. 
\end{proof}

We now make a more specific choice of the elements $a, b, g$, $\alpha, \beta, \gamma$, and $\delta$. 
Since $f \in J$ and $f \not\in \m^2$, writing $\m = (f,h)$ we have 
$$
J=(f, h^q)
$$
for some $q > 0$. Let $g=h^q$. If $J=\m$, then $I = Q :\m$ and  by \cite[Theorem 4.1]{GMTY1} $\calR$ is an  almost Gorenstein graded ring. In what follows, let us  suppose that $J \neq \m$; hence $q > 1$. We may assume that $a \in \m^2 \setminus \m^3$, since $\rmo(Q)= \rmo(I) = 2$ (because $I^2=QI$). 
 Notice that $\alpha, \beta, \gamma, \delta \in J$ (see equations $(***)$ in the proof of Lemma \ref{3.1}). We then have $a, b, c \in J^2$ by equations $(*)$ and $(**)$, so that $I \subseteq J^2$. Therefore since $a \in J^2 \setminus \m^3$, we may assume that the elements $a$ and $b$ have the form 
$$
a = f^2 + \varphi fg + \psi g^2 \ \ \text{and}\ \  b= \varphi' fg + \psi' g^2
$$
with $\varphi, \psi, \varphi', \psi' \in R$. Thus re-choosing $\alpha = f$, $\beta = \varphi f + \psi g$, $\gamma=0$, and $\delta= \varphi' f + \psi' g$, we get  assertion (1) in the following.

\begin{lem}\label{abc}
\begin{enumerate}[$(1)$]
\item 
$a = f^2 + \beta g, ~ b= \delta g, ~\text{and}~c=\delta f$.
Hence $\delta J \subseteq I$.
\item $
\m J = \m (f -g)+(f-h)J.
$
\end{enumerate}
\end{lem}

\begin{proof}(2) Let $\fka = \m (f -g)+(f-h)J$. Since
$$
fh(1-h^{q-1}) = (fh-h^{q+1}) - (fh^q - h^{q+1}) \in \fka,
$$
we have $fh \in \fka$. Therefore $\fka = \m J$ because $\fka + (fh) = \m J$.
\end{proof}

\begin{prop}\label{3.4}
Suppose that $\beta \in \m J$. Then $\calR$ is an almost Gorenstein graded ring.\end{prop}

\begin{proof} We set $\fka = I(f-g)+aJ$. It suffices by Theorem \ref{2.5} to show that $IJ = \fka$. Since $IJ = \fka + (bf, cg)$, we have only to check that $bf, cg \in \fka$. Because $\beta \in \m J$ and $\delta J \subseteq I$ by Lemma \ref{abc} (1), we may write $\delta \beta = a x + b y + cz $ with $x, y, z \in \m$. Hence 
$$
\delta \beta g \equiv bg y + c g z \equiv \delta fg (y+z)  \ \mod \ \fka
$$
because  $bf \equiv b g\ \mod \ \fka$, $b = \delta g$, and $c =\delta f$ by Lemma \ref{abc} (1). 
On the other hand, since $a = f^2 + \beta g$, we have 
$$
\delta \beta g \equiv -\delta f^2 \equiv -\delta fg \ \mod \ \fka
$$
because $\delta f^2 = cf$ and $cg= \delta fg$.
Therefore
$$
\delta fg (1 + y + z) \in \fka,
$$
whence $\delta fg = cg = bf \in \fka$. Thus $IJ = \fka$.
\end{proof}

Let $\overline{I}$ denote the integral closure of $I$. Then $\overline{I}$ is a contracted ideal and $\overline{I}^2 = Q\overline{I}$. Let $K=Q : \overline{I}$. Then $\rmo(K) = 1$ by Proposition \ref{2.8}. Let us choose an element $\zeta \in K \setminus \m^2$. Then since $K = Q : \overline{I} \subseteq Q:I =J$, we have
$$
\zeta = x f + y h^q
$$
with $x, y \in R$, where $x \notin \m$ because $\zeta \not\in \m^2$. Therefore we have 
$$
\m = (\zeta, h), ~~J=(\zeta, h^q), \ \ \text{and}\ \ ~~K=(\zeta, h^{\ell}) \ \ \text{with} \ \ \ell \ge q.
$$
Hence without loss of generality we may assume that $f = \zeta \in K$. Let  $n=\ell -q$. Then taking $g' = h^\ell$, by Lemma \ref{abc} (1) we may furthermore assume that the elements $a$ and $b$ have the form 
$$
a = f^2 + \beta' g'\ \ \text{and}\ \ b= \delta' g'
$$
where $\beta', \delta' \in R$ such that $\beta'h^n = \beta, \delta' h^n = \delta$.  We set $c' = \delta' f$. Hence $\overline{I} = Q + (c') = (a,b,c')$ and we have  the following.

\begin{prop}\label{length}
$\ell_R(\overline{I}/I) = n$.
\end{prop}

\begin{proof}
We have $(f, h^n) \overline{I} \subseteq I$, since $h^n c' = \delta f = c$. Hence  $\overline{I}/I$ is a cyclic module over $R/(f,h^n)$ generated by the image of $c'$.  Let $0 \le i < n$ be an integer and suppose that 
$
h^i c' \in I = Q+(h^nc').
$
Then $h^ic' \in Q$, since $h^ic'(1 - x h^{n-i}) \in Q$ for some $x \in R$. Therefore $h^i \in Q : c' = K = (f, h^{\ell})$, which is impossible because $f,h$ is a regular system of parameters of $R$ and $0 \le i < n= \ell -q < \ell$. Thus $\overline{I}/I \cong R/(f,h^n)$ and hence $\ell_R(\overline{I}/I)= n$.
\end{proof}

To prove Theorem \ref{1.3} we need the following general result.

\begin{lem}\label{o(I)=2} Let $\fka$ be an arbitray $\m$-primary ideal of $R$. Assume that $\fka$ contains a parameter ideal $\q = (a,b)$ of $R$ as a reduction. If $\rmo(\fka) \le 2$, then $\fka^2 = \q \fka$.
\end{lem}

\begin{proof}
Notice that $\rmo(\q)=\rmo(\fka)$. If $\rmo(\fka) = 1$, $\fka$ is generated by two elements, whence $\q= \fka$. Suppose that $\rmo(\fka)=2$ and consider the integral closure $\overline{\fka}$ of $\fka$. Then $\overline{\fka}^2 =\q \overline{\fka}$. Let $K = \q:\overline{\fka}$ and $J = \q:\fka$. Then $\rmo(K)=1$ by Proposition \ref{2.8} because $\overline{\fka}$ is contracted, so that $K = (f,h^\ell)$ and $J = (f,h^q)$ for some regular system $f,h$ of parameters of $R$, where $\ell \ge q$, $\ell > 0$, and $q \ge 0$. We may assume $q > 0$, because $\q=\fka$ if $q = 0$. Remember that after re-choosing $a, b$, we may assume that 
\begin{eqnarray*}
a &=& f^2 + \alpha h^\ell, \\
b &=&\beta h^\ell
\end{eqnarray*}
for some $\alpha, \beta \in R$ (Lemma \ref{abc}). 
Therefore since 
$\left(\begin{smallmatrix}
a\\
b
\end{smallmatrix}\right)
=\left(\begin{smallmatrix}
f&\alpha\\
0&\beta
\end{smallmatrix}\right)
\left(\begin{smallmatrix}
f\\
h^\ell
\end{smallmatrix}\right),$
setting $c_1 = \beta f$, we get $\overline{\fka} = (a,b,c_1)$. On the other hand,
because 
$\left(\begin{smallmatrix}
a\\
b
\end{smallmatrix}\right)
=\left(\begin{smallmatrix}
f&\alpha h^n\\
0&\beta h^n
\end{smallmatrix}\right)
\left(\begin{smallmatrix}
f\\
h^q
\end{smallmatrix}\right)$
where $n = \ell - q~~(\ge 0)$, setting $c = c_1h^n$, we have $\fka = (a,b,c)$. Hence  $h^n\overline{\fka} \subseteq \fka$. Let us now write $${c_1}^2 = a \varphi + b \psi$$ with $\varphi, \psi \in \overline{\fka}$. Then $c^2 = ({c_1h^n})^2 = a(h^{2n}\varphi)+b(h^{2n}\psi) \in (a,b)\fka$, whence $\fka^2 = \q \fka$ as claimed.
\end{proof}

We are now ready to complete the proof of Theorem \ref{1.3}.

\begin{proof}[Proof of Theorem \ref{1.3}] We may assume that $\mathrm{o}(I) = 2$. Hence $I^2=QI$ by Lemma  \ref{o(I)=2}. We maintain the same notation as above. By \cite[Theorem 4.1]{GMTY1} and Proposition \ref{3.4} we may assume that $J \neq \m$ and $\beta \not\in \m J$. Let us write $\beta = \varepsilon f + \rho g$ with $\varepsilon, \rho \in R$. Suppose  that $\varepsilon \not\in \m$. If $n >0$, then since $\beta = \beta' h^n \in (h)$ and $\rho g = \rho (g' h^n) \in (h)$, we get $f \in (h)$. This is impossible since $\m =(f,h)$. Thus $n=0$ and $\overline{I} = I$ by Proposition \ref{length}, so that $\calR$ is an almost Gorenstein graded ring by \cite[Theorem 1.3]{GMTY1}.

Suppose that $\varepsilon \in \m$. Then $\rho \not\in \m$ as $\beta \not\in \m J$. Therefore $$J = (f,g) \subseteq (f,\beta) \subseteq J$$ since $g \in (f,\beta)$. Hence  $J = (f,\beta)=(\xi,\eta)$. Consequently, $\det \left(\begin{smallmatrix}
\alpha_1 & \beta_1 \\
\alpha_2 & \beta_2
\end{smallmatrix}\right) =\alpha_1\beta_2 - \alpha_2\beta_1$ is a unit of $R$, because  
$$
\begin{pmatrix}
-f \\
\beta
\end{pmatrix}
=u
\begin{pmatrix}
\alpha_1 & \beta_1 \\
\alpha_2 & \beta_2
\end{pmatrix}
\begin{pmatrix}
-\eta \\
\xi
\end{pmatrix}
$$
for some unit $u$ of $R$ (see equations $(***)$). On the other hand,   by Lemma \ref{3.1} $\det \left(\begin{smallmatrix}
\delta_1 & -\gamma_1 \\
-\beta_1 & \alpha_1
\end{smallmatrix}\right)=\alpha_1 \delta_1 - \beta_1 \gamma_1$ is a unit of $R$, because $$
\begin{pmatrix}
g \\
-f
\end{pmatrix}
=(-u)
\begin{pmatrix}
\delta_1 & -\gamma_1 \\
-\beta_1 & \alpha_1
\end{pmatrix}
\begin{pmatrix}
\xi \\
\eta
\end{pmatrix}
$$
by equations $(***)$. We are interested in the form of matrices equivalent to ${}^t\Bbb M$. We write 
$$
(-u)
\begin{pmatrix}
\delta_1 & -\gamma_1 \\
-\beta_1 & \alpha_1
\end{pmatrix}
{}^t {\Bbb M}
=
\begin{pmatrix}
g & G_1 & G_2\\
-f & u H_1 & u H_2
\end{pmatrix}
$$
where
\begin{eqnarray*}
H_1 &=&  (\alpha_2 \beta_1 - \alpha_1 \beta_2)T_2 + (\alpha_3 \beta_1 -\alpha_1 \beta_3)T_3 \ \  \text{and} \\
H_2 &=&  (\beta_1 \gamma_1 - \alpha_1 \delta_1)T_1 + (\beta_1 \gamma_2 -\alpha_1 \delta_2)T_2 + (\beta_1 \gamma_3 - \alpha_1 \delta_3)T_3.
\end{eqnarray*}
Since $\alpha_2 \beta_1 - \alpha_1 \beta_2 \notin \m$ and $\beta_1 \gamma_1 - \alpha_1 \delta_1 \not\in \m$, after elementary column operations with coefficients in $R$ on the matrix  
$
\left(\begin{smallmatrix}
g & G_1 & G_2\\
-f & u H_1 & u H_2
\end{smallmatrix}\right)
$, we may assume that $uH_1 \equiv T_2 \ \mod \ (T_3)$ and that $uH_2 \equiv T_1 \ \mod \ (T_2, T_3)$.   Hence the matrix 
$
\left(\begin{smallmatrix}
g & G_1 & G_2\\
-f & u H_1 & u H_2
\end{smallmatrix}\right)$
is equivalent to a matrix of the form
$
\left(\begin{smallmatrix}
-g& m_1 & m_2 \\
f & \ell_1 & \ell_2
\end{smallmatrix}\right)
$
where $m_1, m_2, \ell_1, \ell_2$ are linear forms in $S$ such that $(f, \ell_1, \ell_2, T_3) = (f,T_1, T_2, T_3)$. Therefore by Proposition \ref{res} $$\rmK_{\calR}/\calR{\cdot}\theta \cong \left[S/(f,\ell_1,\ell_2)\right](-1)$$ for some $\theta \in \left[\rmK_{\calR}\right]_1$, whence by Definition \ref{1.2} $\calR$ is  an almost Gorenstein graded ring. This completes the proof of Theorem \ref{1.3}.
\end{proof}


\section{Proof of Theorem \ref{1.4}}
The purpose of this section is to prove Theorem \ref{1.4}.

Let $(R, \m)$ be a two-dimensional regular local ring with infinite residue class field. Let $x,y$ be a regular system of parameters of $R$. Let $n \ge 3$ be an integer and put $Q=(x^3, y^n)$. We choose integers $\alpha$, $\beta$ so that $0 < \alpha < \beta < n$ and set 
$$
I=(x^3, x^2y^{\alpha}, xy^{\beta}, y^n) \ ~~ \ \text{and}~~\  J=(x^2, xy^{n-\beta}, y^{n-\alpha}).
$$
Then $I$ and $J$ are contracted ideals of $R$  and
$$
J=Q:I, ~~I:\m=I:y,~~\ \text{and}~~\ J:\m = J:y.
$$
We have  $\mu_R(IJ)= \mathrm{o}(IJ) + 1 =6$ and $\mu_R(\m J) = \mathrm{o}(\m J) +1 =4$. If $\beta + 2 \le 2\alpha$, then $x^2y^{\alpha -1}$ is integral over $I$ but $x^2y^{\alpha -1} \not\in I$, so that $I$ is not integrally closed.

A direct computation shows the following.

\begin{prop}
$I^2=QI$ if and only if $\beta \le 2 \alpha$, $n \le \alpha + \beta$, and $n + \alpha \le 2\beta$.
\end{prop}

We then have the following. The simplest example satisfying the conditions in Theorem \ref{4.2} which are not integrally closed is the ideal $I = (x^3, x^2y^3, xy^4,y^5)$.

\begin{thm}\label{4.2}
Suppose that $I^2 = QI$ and $n + \alpha = 2 \beta$. Then $\calR(I)$ is an almost Gorenstein graded ring.
\end{thm}

\begin{proof}
We have $\alpha \ge n- \beta$, $\beta \ge n- \alpha$, and $2n - \beta = n - \alpha + \beta$, because $I^2 =QI$ and $n + \alpha= 2 \beta$. Therefore
\begin{eqnarray*}
IJ &=& (x^5, x^4y^{n-\beta}, x^3y^{n-\alpha}, x^2y^n, xy^{n-\alpha + \beta}, y^{2n - \alpha}) \\ 
&=& I(x^2-y^{n-\alpha}) + x^3 J.
\end{eqnarray*}
On the other hand, because $n -\alpha \ge n- \beta + 1$ we get
\begin{eqnarray*}
\m J &=& (x^3, x^2y, xy^{n-\beta + 1}, y^{n - \alpha + 1}) \\ 
&=& \m (x^2-y^{n-\alpha}) + y J.
\end{eqnarray*}
Hence $\calR(I)$ is an almost Gorenstein graded ring by Theorem \ref{2.5}.
\end{proof}

We are now in a position to prove Theorem \ref{1.4}, which shows that contrary to Theorem \ref{4.2}, the Rees algebra $\calR (I)$ is not an almost Gorenstein graded ring,  if $n < \alpha + \beta$, $n + \alpha < 2\beta$, and $\beta < 2 \alpha$.

\begin{proof}[Proof of Theorem \ref{1.4}] 
Suppose that $\calR = \calR(I)$ is an almost Gorenstein graded ring and choose an exact sequence
$$
0 \to \calR \overset{\varphi}{\longrightarrow} \rmK_{\calR}(1) \to C \to 0\ \ \ \ \ \ (E)
$$
of graded $\calR$-modules such that either $C = (0)$ or $C \ne (0)$ and $C_{\fkM}$ is an Ulrich $\calR_{\fkM}$-module, where $\fkM = \m \calR + \calR_+$. Because $\rmK_{\calR}(1)= J \calR$ (Proposition \ref{2.4a}) and $\mu_R(J) = 3$, $\calR$ is not a Gorenstein ring, whence $C \ne (0)$. We set $h=\varphi (1) \in J$. Then $h \not\in \fkM \rmK_{\calR}$ (\cite[Corollary 3.10]{GTT}) so that $\mu_{\calR}(C) = 2$. Therefore  $\mu_{\calR}(\fkM C) \le 4$, because $\fkM C_{\fkM} = (\xi, \eta) C_{\fkM}$  for some  $\xi, \eta \in \fkM$ (\cite[Proposition 2.2 (2)]{GTT}).  Let  $X = \fkM J / [\fkM h + \fkM^2 J] ~(\cong \fkM C)$ and notice that $X$ is generated by elements of degrees $0$ and $1$. Then because
$$
X_0 \cong \m J/ [\m h + \m^2 J] ~~\ \text{and}~~\ X_1 \cong IJ/ [I h + \m I J]
$$
(here $X_i$ denotes the homogeneous component of degree $i$), we have
$$
\mu_R(\m J/ \m h) + \mu_R(IJ / I h)   = \mu_{\calR}(X) =\mu_{\calR}(\fkM C)\le 4,
$$
while we get $\mu_R(\m J/ \m h) \ge 2$ and $\mu_R(IJ / I h) \ge 2$, because $\mu_R(\m J) =4$ and $\mu_R(IJ) = 6$. Hence $\mu_R(IJ / I h) = 2$. Let us write $h= a x^2 + b xy^{n-\beta} + c y^{n-\alpha}$ with $a, b, c \in R$ (remember that $h \in J = (x^2, xy^{n-\beta}, y^{n-\alpha}$)).  We set $
V = [I h + \m IJ] / \m IJ 
$. Then it is direct to check that the $R/\m$-space $V$ 
is spanned by the images of the following four elements in $Ih$ 
$$
ax^5 + b x^4y^{n-\beta} + c x^3 y^{n-\alpha},~ c x^2y^n, ~b x^2y^n, \quad\text{and} \quad a x^2y^n + b xy^{2n-\beta} + c y^{2n-\alpha}.
$$ 
 However because $\mu_R(IJ / I h) = 2$ and $\mu_R(IJ) = 6$, we must have $\ell_R(V) = 4$, which is clearly impossible. Thus $\calR$ is not an almost Gorenstein graded ring.
\end{proof}


\begin{remark}
There are contracted ideals $I$ with order greater than $3$, whose Rees algebras $\calR(I)$ are not almost Gorenstein graded rings. For example, let $m \ge 4$ be an integer and set 
\begin{center}
$Q=(x^m, y^{2m})$ \ and \ $I=Q+(x^{m-i}y^{2i+1} \mid 1 \le i \le m-1)$.
\end{center}
Then $I^2=QI$ and $I$ is a contracted ideal with $\mathrm{o}(I) = m$. One can show similarly as Theorem \ref{1.4} that $\calR(I)$ is not an almost Gorenstein graded ring.
\end{remark}

\section{Analysis of certain non-contracted ideals}
Let $(R, \m)$ be a two-dimensional regular local ring with infinite residue class field. Let $x,y$ be a regular system of parameters of $R$. We close this paper with the analysis of the following ideal $I$. Notice that $\rmo(I)=3$ but $I$ is not a contracted ideal, if $n \ge 3$.

\begin{thm}\label{5.1}
Choose integers $\alpha, n$ so that $0 < \alpha < n$, $2 \alpha \ge n$ and set $I=(x^3, x^2y^{\alpha}, y^n)$, $Q=(x^3, y^n)$, and  $J = Q:I$. Then $I^2 = QI$ and $J = (x, y^{n-\alpha})$. We furthermore have the following.
\begin{enumerate}[$(1)$]
\item If $2 \alpha =n$, then $\calR(I)$ is an almost Gorenstein graded ring.
\item If $2 \alpha >n$, then $\calR(I)$ is not an almost Gorenstein graded ring.\end{enumerate}
\end{thm}

\begin{proof}
$(1)$ We have $\m J = \m x + y J$ and $IJ = Ix +y^n J$. Hence by Theorem \ref{2.5} $\calR(I)$ is an almost Gorenstein graded ring.

$(2)$ Suppose that $\calR = \calR(I)$ is an almost Gorenstein graded ring and choose the exact sequence
$$
0 \to \calR \overset{\varphi}{\longrightarrow} \rmK_{\calR}(1) \to C \to 0
$$
of graded $\calR$-modules such that either $C = (0)$ or $C \ne (0)$ and $C_{\fkM}$ is an Ulrich $\calR_{\fkM}$-module, where $\fkM = \m \calR + \calR_+$. 
We actually have $\mu_{\calR}(C)=1$, since $\rmK_{\calR}(1) = J\calR =(x, y^{n-\alpha}) \calR$ and $\varphi(1) \not\in \fkM \rmK_{\calR}$ by \cite[Corollary 3.10]{GTT}. Therefore  $\mu_{\calR}(\fkM C) \le2$ by \cite[Proposition 2.2 (2)]{GTT}. 
On the other hand, $\mu_R(IJ) = 4$ and $\mu_R(\m J) = 3$, since
$$
IJ = (x^4, x^3y^{n-\alpha}, xy^n, y^{2n-\alpha})~~\ \ \text{and}\ \ \ ~~\m J = (x^2, xy, y^{n-\alpha + 1}).
$$
Consequently for the same reason as in the proof of Theorem \ref{1.4} we get $\mu_R(IJ/I h) = \mu_R(\m J/\m h) = 1$. Hence $\ell_R([Ih + \m IJ]/\m IJ) = 3$. Nevertheless, writing $h=ax + by^{n-\alpha}$ with $a, b\in R$, we see that the $R/\m$-space $V = [Ih + \m IJ]/\m IJ$ is spanned by the images of the following elements in $Ih$
$$ax^4 + bx^3y^{n-\alpha}, \ \ ax^3y^{\alpha} + bx^2y^n, \ \ \text{and}\ \ axy^n + b y^{2n-\alpha},
$$
so that the dimension of $V$ is at most two, because $x^3y^{\alpha}, x^2y^n \in \m IJ$ (remember that $2 \alpha > n$). This is a contradiction. 
\end{proof}


To end this paper let us note a question. 
Let $(R,\m)$ be a two-dimensional regular local ring with infinite residue class field. Let $\{I_i\}_{1 \le i \le \ell}$ be a finite family of $\m$-primary ideals of $R$. With this notation we pose the following question, which is wildly open at this moment.

\begin{ques}

When is $\calR (I_1I_2 \ldots I_\ell)$ an almost Gorenstein graded ring?

\end{ques}



\begin{thebibliography}{20}

\bibitem{BF}
{\sc V. Barucci and R. Fr\"{o}berg}, One-dimensional almost Gorenstein rings, {\em J. Algebra}, {\bf 188} (1997), no. 2, 418--442.


\bibitem{BHU}
{\sc J. P. Brennan, J. Herzog, and B. Ulrich}, Maximally generated maximal Cohen-Macaulay modules, {\em Math. Scand.}, {\bf 61} (1987), no. 2, 181--203.

\bibitem{BH}
{\sc W. Bruns and J. Herzog}, Cohen-Macaulay rings, {\em Cambridge studies in advanced mathematics}, {\bf 39}, 1993.



\bibitem{G}
{\sc S. Goto}, Integral closedness of complete-intersection ideals, {\em J. Algebra}, {\bf 108} (1987), 151--160.






\bibitem{GMP}
{\sc S. Goto, N. Matsuoka, and T. T. Phuong}, Almost Gorenstein rings, {\em J. Algebra}, {\bf 379} (2013), 355--381.

\bibitem{GMTY1}
{\sc S. Goto, N. Matsuoka, N. Taniguchi, and K.-i. Yoshida}, The almost Gorenstein Rees algebras over two-dimensional regular local rings, {\em J. Pure Appl. Algebra} (to appear).

\bibitem{GMTY2}
{\sc S. Goto, N. Matsuoka, N. Taniguchi, and K.-i. Yoshida}, The almost Gorenstein Rees algebras of parameters, {\em J. Algebra}, {\bf 452} (2016), 263-278.

\bibitem{GS}
{\sc S. Goto and Y. Shimoda}, On the Rees algebras of Cohen-Macaulay local rings, Commutative algebra (Fairfax, Va., 1979), 201--231, {\em Lecture Notes in Pure and Appl. Math.}, {\bf 68}, Dekker, New York, 1982.

\bibitem{GTT}
{\sc S. Goto, R. Takahashi, and N. Taniguchi}, Almost Gorenstein rings -towards a theory of higher dimension, {\em J. Pure Appl. Algebra}, {\bf 219} (2015), 2666--2712.





\bibitem{H}
{\sc J. Herzog}, Certain complexes associated to a sequence and a matrix, {\em Manuscripta Math.}, {\bf 12} (1974), 217--248.


\bibitem{HK}
{\sc J. Herzog and E. Kunz}, Dear kanonische Modul eines-Cohen-Macaulay-Rings, {\em Lecture Notes in Mathematics}, {\bf 238}, Springer-Verlag, 1971.



\bibitem{Hu}
{\sc C. Huneke}, Complete ideals in two-dimensional regular local rings, {\em MSRI Publications}, {\bf 15} (1989), 325--338.




\bibitem{HSa}
{\sc C. Huneke and J. D. Sally}, Birational extensions in dimension two and integrally closed ideals, {\em J. Algebra}, {\bf 115} (1988), 481-450.










\bibitem{L}
{\sc J. Lipman}, On complete ideals in regular local rings, in  "Algebraic Geometry and Commutative Algebr" in honor of M. Nagata, 1987, 203-231.




\bibitem{MP}
{\sc  F. Mui\~{n}os and  F. Planas-Vilanova}, The equations of Rees algebras of equimultiple ideals of deviation one, {\em Proc. Amer. Math. Soc.}, {\bf 141} (2013), 1241-1254.





\bibitem{R}
{\sc D. Rees}, Hilbert functions and pseudo-rational local rings of dimension two, {\em J. London Math. Soc.}, {\bf 24} (1981), 467--479.






\bibitem{S}
{\sc J. D. Sally}, Cohen-Macaulay local rings of maximal embedding dimension, {\em J. Algebra}, {\bf 56} (1979), 168--183. 



\bibitem{SH}
{\sc I. Swanson and C. Huneke}, Integral Closure of Ideals, Rings, and Modules, {\em Cambridge University Press}, 2006.






\bibitem{U} 
{\sc B. Ulrich}, Ideals having the expected reduction number, {\em Amer. J. Math.}, {\bf 118} (1996), no. 1, 17--38.

\bibitem{V1}
{\sc J. K. Verma}, Joint reductions and Rees algebras, {\em Math. Proc. Camb. Phil. Soc.}, {\bf 109} (1991), 335--342.

\bibitem{V2}
{\sc J. K. Verma}, Rees algebras of contracted ideals in two-dimensional regular local rings, {\em J. Algebra}, {\bf 141} (1991), 1--10.



\bibitem{Z}
{\sc O. Zariski}, Polynomial ideals defined by infinitely near base points, {\em Amer. J. Math.}, {\bf 60} (1938), 151--204.

\bibitem{ZS}
{\sc O. Zariski and P. Samuel}, Commutative Algebra Volume II, {\em Springer}, 1960.

\end{thebibliography}
\end{document}